\documentclass[11pt,twoside, leqno]{article}

\usepackage{amssymb}
\usepackage{amsmath}
\usepackage{amsthm}
\usepackage{bbding}

\usepackage{mathrsfs}

\allowdisplaybreaks

\def\dsum{\displaystyle\sum}
\def\dint{\displaystyle\int}

\def\r{\right}
\def\lf{\left}
\def\pat{\partial}
\def\ls{\lesssim}

\def\fz{\infty}
\def\fz{\infty}
\def\az{\alpha}
\def\supp{{\mathop\mathrm{\,supp\,}}}

\def\loc{{\mathop\mathrm{\,loc\,}}}
\def\bmo{{\mathop\mathrm{BMO}}}
\def\dist{{\mathop\mathrm{dist}}}
\def\mol{{\mathop\mathrm{mol}}}

\def\lz{\lambda}
\def\dz{\delta}
\def\bdz{\Delta}
\def\ez{\epsilon}

\def\bz{\beta}
\def\fai{\varphi}

\def\tz{\theta}
\def\sz{\sigma}
\def\wz{\widetilde}

\def\ls{\lesssim}

\def\pat{\partial}

\def\rr{{\mathbb R}}
\def\rn{{{\rr}^n}}
\def\zz{{\mathbb Z}}
\def\nn{{\mathbb N}}

\def\zz{{\mathbb Z}}
\def\nn{{\mathbb N}}
\def\cc{{\mathbb C}}

\def\cl{{\mathcal L}}

\def\cm{{\mathcal M}}

\newtheorem{theorem}{Theorem}[section]
\newtheorem{lemma}{Lemma}[section]
\newtheorem{corollary}{Corollary}[section]

\theoremstyle{definition}

\newtheorem{definition}{Definition}[section]
\numberwithin{equation}{section}

\def\hs{\hspace{0.3cm}}

\begin{document}

\arraycolsep=1pt

\title{\Large\bf Endpoint Boundedness of Riesz Transforms on
Hardy Spaces Associated with Operators \footnotetext
{\hspace{-0.7cm} J. Cao $\cdot$ D.-C. Yang (\Envelope) $\cdot$
S.-B. Yang\\
\noindent School of Mathematical Sciences, Beijing Normal
University, Laboratory of Mathematics and Complex Systems, Ministry
of Education, Beijing 100875, People's Republic of China\\
\bigskip
\noindent e-mails: dcyang@bnu.edu.cn\\
\noindent J. Cao\\
\bigskip
e-mails: caojun1860@mail.bnu.edu.cn\\
S.-B. Yang\\
e-mails: yangsibei@mail.bnu.edu.cn}}
\author{Jun Cao $\cdot$ Dachun Yang $\cdot$ Sibei Yang}
\date{}
\maketitle

\noindent{\bf Abstract}\quad Let $L_1$ be a nonnegative self-adjoint
operator in $L^2({\mathbb R}^n)$ satisfying the Davies-Gaffney
estimates and $L_2$ a second order divergence form elliptic operator
with complex bounded measurable coefficients. A typical example of
$L_1$ is the Schr\"odinger operator $-\Delta+V$, where $\Delta$ is
the Laplace operator on ${\mathbb R}^n$ and $0\le V\in
L^1_{\mathop\mathrm{loc}} ({\mathbb R}^n)$. Let
$H^p_{L_i}(\mathbb{R}^n)$ be the Hardy space associated to $L_i$ for
$i\in\{1,\,2\}$. In this paper, the authors prove that the Riesz
transform $D (L_i^{-1/2})$ is bounded from $H^p_{L_i}(\mathbb{R}^n)$
to the classical weak Hardy space $WH^p(\mathbb{R}^n)$ in the
critical case that $p=n/(n+1)$. Recall that it is known that $D
(L_i^{-1/2})$ is bounded from $H^p_{L_i}(\mathbb{R}^n)$ to the
classical Hardy space $H^p(\mathbb{R}^n)$ when $p\in(n/(n+1),\,1]$.


\bigskip


\section{Introduction}\label{s1}

\hskip\parindent The Hardy spaces, as a suitable substitute of
Lebesgue spaces $L^p(\rn)$ when $p\in (0,1]$,
play an important role in various fields
of analysis and partial differential equations. For example, when
$p\in(0,\,1]$, the \emph{Riesz transform} $\nabla(-\bdz)^{-1/2}$ is not
bounded on $L^p(\rr^n)$, but bounded on the Hardy space
$H^p(\rr^n)$, where $\Delta$ is the \emph{Laplacian operator}
$\sum_{i=1}^n\frac {\partial^2}{\partial x_i^2}$
and $\nabla$ is the \emph{gradient operator} $(\frac \partial{\partial x_1},
\cdots, \frac \partial{\partial x_n})$ on $\rn$. It is well
known that the classical Hardy spaces
$H^p(\rn)$ are essentially related to $\bdz$, which has been
intensively studied in, for example, \cite{cw77,fs,s,sw,tw}
and their references.

In recent years, the study of Hardy spaces associated to
differential operators inspires great interests; see, for example,
\cite{adm2,amr08,ar,dy05-2,dz1,dz2,hlmmy,hm09,hm09c,hmm,dxy07} and their
references. In particular, Auscher, Duong and McIntosh \cite{adm2}
first introduced the Hardy space $H_L^1(\rn)$ associated to $L$,
where the \emph{heat kernel generated by $L$ satisfies a pointwise Poisson
type upper bound}. Later, Duong and Yan \cite{dy05-1,dy05-2}
introduced its dual space $\bmo_L(\rn)$ and established the dual
relation between $H_L^1(\rn)$ and $\bmo_{L^*}(\rn)$, where $L^*$
denotes the \emph{adjoint operator} of $L$ in $L^2(\rn)$. Yan \cite{y08}
further introduced the Hardy space $H_L^p(\rn)$ for some
$p\in(0,\,1]$ but near to $1$ and generalized these results to $H_L^p(\rn)$ and
their dual spaces. A theory of the Orlicz-Hardy space and its dual space
associated to a such $L$ were developed in \cite{jyz09,jya}.

Moreover, for the \emph{Schr\"odinger operator} $-\bdz+V$, Dziuba\'nski and
Zienkiewicz \cite{dz1,dz2} first introduced the Hardy spaces
$H_{-\bdz+V}^p(\rr^n)$ with the \emph{nonnegative potential} $V$ belonging
to the reverse H\"older class $B_q(\rr^n)$ for certain $q\in
(1,\fz)$. As a special case, the Hardy space
$H^p_{-\bdz+V}(\rn)$ associated with
$-\bdz+V$ with $0\le V\in L^1_{\mathop\mathrm{loc}} ({\mathbb R}^n)$
and $p\in (0,1]$ but near to $1$ was also
studied in, for example, \cite{dy05-2,hlmmy,y08,jyz09,yyz,yz,jy,dl}.
More generally, for \emph{nonnegative self-adjoint operators
$L$ satisfying the Davies-Gaffney estimates}, Hofmann et al.
\cite{hlmmy} introduced a new Hardy space $H_L^1(\rn)$. In
particular, when $L\equiv-\bdz+V$ with $0\le V\in
L^1_{\mathop\mathrm{loc}} ({\mathbb R}^n)$, Hofmann et al. originally
showed that the Riesz transform $\nabla(L^{-1/2})$ is bounded
from $H_L^1(\rn)$ to the classical Hardy space $H^1(\rn)$.
These results in \cite{hlmmy} were further extended to the
Orlicz-Hardy space and its dual space in \cite{jy}. In particular, as a
special case of \cite[Theorem 6.3]{jy}, it was proved that
$\nabla(-\bdz+V)^{-1/2}$ with $0\le V\in L^1_{\loc}(\rn)$ is bounded
from the Hardy space $H^p_{-\bdz+V}(\rn)$ to $H^p(\rn)$ if
$p\in(\frac{n}{n+1},1]$.

Also, Auscher and Russ \cite{ar} studied the Hardy space $H^1_L$ on
strongly Lipschitz domains associated with a \emph{second order divergence
form elliptic operator} $L$ whose heat kernels have the Gaussian
upper bounds and certain regularity. Hofmann and Mayboroda \cite{hm09,hm09c}
and Hofmann et al. \cite{hmm} introduced the Hardy and
Sobolev spaces associated to a \emph{second order divergence form elliptic
operator $L$ on $\rn$ with complex bounded measurable coefficients}.
Notice that, for the second order divergence form elliptic operator
$L$, the kernel of the heat semigroup may fail to satisfy the
Gaussian upper bound estimate and, moreover, $L$ may not be
nonnegative self-adjoint in $L^2(\rn)$. Hofmann et al.
\cite{hmm} also proved that the associated Riesz transform $\nabla
{L^{-1/2}}$ is bounded from $H^p_L(\rn)$ to the classical Hardy
space $H^p(\rn)$ with $p\in(\frac{n}{n+1},\,1]$, which was also
independently obtained by Jiang and Yang in \cite[Theorem 7.4]{jy10}.
Moreover, a theory of the Orlicz-Hardy space and its dual
space associated to $L$ were developed in \cite{jy10, jy10a}.

Recently, the Hardy space $H_{(-\bdz)^2+V^2}^1(\rn)$ associated to
the Schr\"odinger-type operators $(-\bdz)^2+V^2$ with $0\le V$
satisfying the reverse H\"older inequality was also studied in
\cite{cly}. Moreover, the Hardy space $H_{L}^p(\rn)$ associated to a
\emph{one-to-one operator of type $\omega$ satisfying the
$k$-Davies-Gaffney estimate and having a bounded $H_{\fz}$ functional
calculus} was  introduced in \cite{cy}, where $k\in\nn$. Notice that when
$k=1$, the $k$-Davies-Gaffney estimate is just the Davies-Gaffney
estimate. Typical examples of such operators include the \emph{$2k$-order
divergence form homogeneous elliptic operator $T_1$ with complex
bounded measurable coefficients} and the \emph{$2k$-order
Schr\"odinger-type operator $T_2\equiv(-\Delta)^k+V^k$}, where $0\le
V\in L^k_{\mathop\mathrm{loc}}(\mathbb{R}^n)$. It was further proved
that the associated Riesz transform $\nabla^k {T_i^{-1/2}}$ for
$i\in\{1,2\}$ is bounded from $H^p_{T_i}(\rn)$ to $H^p(\rn)$ with
$p\in(\frac{n}{n+k},\,1]$ in \cite{cy}.

On the other hand, the weak Hardy space $WH^1(\rn)$ was first
introduced by Fefferman and Soria in \cite{fs1}. Then, Liu \cite{li}
studied the weak $WH^p(\rn)$ space for $p\in(0,\,\fz)$ and established
a weak atomic decomposition for $p\in(0,\,1]$. Liu in \cite{li} also
showed that the $\dz$-Calder\'on-Zygmund operator is bounded
from $H^p(\rn)$ to $WH^p(\rn)$ with $p=n/(n+\dz)$, which was
extended to the weighted weak Hardy spaces in \cite{qy}.

Let $L_1$ be a \emph{nonnegative self-adjoint operator in
$L^2(\rn)$ satisfying the Davies-Gaffney estimates} and
$L_2$ a \emph{second order divergence form elliptic operator
with complex bounded measurable
coefficients}. A typical example of $L_1$ is the Schr\"odinger
operator $-\Delta+V$, where $0\le V\in L^1_{\mathop\mathrm{loc}} ({\mathbb
R}^n)$. Let $H^p_{L_i}(\mathbb{R}^n)$ be the \emph{Hardy space associated
to $L_i$} for $i\in\{1,\,2\}$. In this paper, we prove that
the Riesz transform $D (L_i^{-1/2})$ is bounded from
$H^p_{L_i}(\mathbb{R}^n)$ to the weak Hardy space
$WH^p(\mathbb{R}^n)$ in the critical case that $p=n/(n+1)$.
To be precise, we have the following general result.

\begin{theorem}\label{t1.1}
Let $p\equiv n/(n+1)$, $L_1$ be a nonnegative self-adjoint operator
in $L^2(\rn)$ satisfying the assumptions (A$_1$) and (A$_2$) as in
Section \ref{s2} and $D$ the operator satisfying the assumptions
(B$_1$), (B$_2$) and (B$_3$) as in Section \ref{s2}. Then the
operator $D (L_1^{-1/2})$ is bounded from $H_{L_1}^p(\rn)$ to the
classical weak Hardy space $WH^p(\rn)$. Moreover, there exists a
positive constant $C$ such that for all $f\in H_{L_1}^p(\rn)$,
\begin{eqnarray*}
\lf\|D (L_1^{-1/2})f\r\|_{WH^p(\rn)}\le C\|f\|_{H_{L_1}^p(\rn)}.
\end{eqnarray*}
\end{theorem}

As an application of Theorem \ref{t1.1},
we obtain the boundedness of $\nabla(-\bdz+V)^{-1/2}$ with
$0\le V\in L^1_\loc(\rn)$ from $H_{-\bdz+V}^p(\rn)$ to the classical weak
Hardy space $WH^p(\rn)$ in the critical case that $p=n/(n+1)$
as follows.

\begin{corollary}\label{c1.1}
Let $p\equiv n/(n+1)$ and $0\le V\in L^1_{\mathop\mathrm{loc}} ({\mathbb R}^n)$.
Then the Riesz transform $\nabla(-\bdz+V)^{-1/2}$ is bounded from
$H_{-\bdz+V}^p(\rn)$ to $WH^p(\rn)$.
Moreover, there exists a positive constant $C$ such that for all
$f\in H_{-\bdz+V}^p(\rn)$,
\begin{eqnarray*}
\lf\|\nabla(-\bdz+V)^{-1/2}f\r\|_{WH^p(\rn)}\le C\|f\|_{H_{-\bdz+V}^p(\rn)}.
\end{eqnarray*}
\end{corollary}

On the Riesz transform defined by the second order divergence form
elliptic operator with complex bounded measurable coefficients,
we also have the following endpoint boundedness in the critical
case that $p\equiv n/(n+1)$.

\begin{theorem}\label{t1.2}
Let $p\equiv n/(n+1)$ and $L_2$ be the second order divergence form
elliptic operator with complex bounded measurable coefficients.
Then the Riesz transform $\nabla  (L_2^{-1/2})$ is bounded from
$H_{L_2}^p(\rn)$ to $WH^p(\rn)$.
Moreover, there exists a positive constant $C$ such that for all
$f\in H_{L_2}^p(\rn)$,
\begin{eqnarray*}
\lf\|\nabla (L_2^{-1/2})f\r\|_{WH^p(\rn)}\le
C\|f\|_{H_{L_2}^p(\rn)}.
\end{eqnarray*}
\end{theorem}

Recall that the second order divergence form
elliptic operator with complex bounded measurable coefficients
may not be nonnegative self-adjoint operator
in $L^2(\rn)$. Thus, we cannot deduce the conclusion
of Theorem \ref{t1.2} from Theorem \ref{t1.1}. However, if $L$ is a
second order divergence form elliptic operator with real symmetric
bounded measurable coefficients, then $L$ satisfies the assumptions
of both Theorem \ref{t1.1} and Theorem \ref{t1.2}.

We prove Theorems \ref{t1.1} and \ref{t1.2}
by using the characterization of $WH^p(\rn)$ in terms of
the radial maximal function, namely, we need estimate the
weak $L^p(\rn)$ quasi-norm of the radial maximal function of
the Riesz transform acting on the atoms or molecules of the Hardy
spaces $H_{L_i}^p(\rn)$. Unlike the proof of the endpoint boundedness
of the classical Riesz transform $\nabla(-\bdz)^{-1/2}$,
whose kernel has the pointwise size estimate and
regularity, the strategy to show Theorems
\ref{t1.1} and \ref{t1.2} is to divide the radial maximal function into two
parts by the time $t$ based on the radius of the associated balls of atoms
or molecules and then estimate each part via using $L^2$ off-diagonal
estimates (see \cite{hm,hmm} or Lemma \ref{l2.1} below).

This paper is organized as follows. In Section \ref{s2}, we
describe some assumptions on the operator $L_1$;
then we recall some
notion and properties concerning the Hardy space associated to $L_1$
and second order divergence form elliptic operator $L_2$ with
complex bounded measurable coefficients. We also recall the
definition of weak Hardy spaces and present some technical lemmas
which are used later in the next section. Section \ref{s3} is
devoted to the proof Theorem \ref{t1.1}, Corollary \ref{c1.1},
and Theorem \ref{t1.2}. In Section \ref{s4},
a similar result on the Riesz transforms
defined by \emph{higher order divergence form homogeneous elliptic operators
with complex bounded measurable coefficients} or
\emph{Schr\"odinger-type operators}
is also presented.

Finally, we make some conventions on the notation. Throughout the whole
paper, we always let $\nn\equiv\{1,2,\cdots\}$ and $\zz_+\equiv
\nn\cup\{0\}$. We use $C$ to denote a {\it positive constant}, that
is independent of the main parameters involved but whose value may
differ from line to line. {\it Constants with subscripts},
such as $C_0$, do not change in
different occurrences. If $f\le Cg$, we then write $f\ls g$; and if
$f\ls g\ls f$, we then write $f\sim g$. For all $x\in\rr^n$ and
$r\in(0,\fz),$ let $B(x,r)\equiv\{y\in\rr^n:|x-y|<r\}$ and
$\az B(x,r)\equiv B(x,\az r)$ for any $\az>0$.  Also, for
any set $E\in \rn$, we use $E^\complement$ to denote the \emph{set}
$\rn\setminus E$ and $\chi_E$ the {\it characteristic function} of $E$.

\section{Preliminaries}\label{s2}

\hskip\parindent We begin with recalling some known
results on the Hardy spaces associated to operators and the weak
Hardy spaces.

Let ${L_1}$ be a \emph{linear operator} initially defined in $L^2(\rn)$ satisfying
the following {\it assumptions}:

(A$_1$) ${L_1}$ is nonnegative self-adjoint;

(A$_2$) The semigroup $\{e^{-t{L_1}}\}_{t>0}$ generated by ${L_1}$
is analytic on $L^2(\rn)$ and satisfying the \emph{Davies-Gaffney
estimates}, namely, there exist positive constants $C_1$ and $C_2$
such that for all closed sets $E$, $F\subset\rn$, $t\in(0,\,\fz)$
and $f\in L^2(\rn)$ supported in $E$,
\begin{eqnarray}\label{2.1}
\|e^{-t{L_1}}f\|_{L^2(F)}\le
C_1\exp\lf\{-\frac{[\dist(E,\,F)]^2}{C_2t}\r\}\|f\|_{L^2(E)},
\end{eqnarray}
where and in what follows, $\dist(E,\,F)\equiv \inf_{x\in E,\,y\in
F}|x-y|$ is the \emph{distance between $E$ and $F$}.

Typical examples of operators satisfying assumptions (A$_1$)
and (A$_2$) include the second order divergence form
elliptic operator with real symmetric bounded
measurable coefficients and the Schr\"odinger operator $-\bdz+V$
with $0\le V\in L^1_\loc(\rn)$.

Let $\Gamma(x)\equiv\{(y,\,t)\in\rn\times(0,\,\fz):\ |x-y|<t\}$ be
the \emph{cone with the vertex} $x\in\rn$. For all $f\in L^2(\rn)$ and
$x\in\rn$, the {\it ${L_1}$-adapted square function} $S_{L_1}f(x)$ is
defined by
\begin{eqnarray*}
S_{L_1}f(x)\equiv\lf\{\iint_{\Gamma(x)}|t^{2}{L_1}e^{-t^{2}{L_1}}f(y)|^2
\frac{dy\,dt}{t^{n+1}}\r\}^{1/2}.
\end{eqnarray*}

As in \cite{hlmmy,jy}, we define the Hardy space $H_{L_1}^p(\rn)$
associated to the operator ${L_1}$ as follows.

\begin{definition}[\cite{hlmmy,jy}]\label{d2.1}
Let $p\in(0,\,1]$ and ${L_1}$ be an operator defined in $L^2(\rn)$
satisfying the assumptions (A$_1$) and (A$_2$). A function $f\in
L^2(\rn)$ is said to be in $\mathbb{H}_{L_1}^p(\rn)$ if $S_{L_1}f\in
L_1^p(\rn)$; moreover, define
$\|f\|_{H_{L_1}^p(\rn)}\equiv\|S_{L_1}f\|_{{L}^p(\rn)}$. The {\it
Hardy space} $H_{L_1}^p(\rn)$ is then defined to be the completion
of $\mathbb{H}_{L_1}^p(\rn)$ with respect to the quasi-norm
$\|\cdot\|_{H_{L_1}^p(\rn)}$.
\end{definition}

For all $p\in(0,\,1]$ and $M\in\nn$, a function $a\in L^2(\rn)$ is
called a \emph{$(p,\,2,\,M)_{L_1}$-atom} if there exists a function
$b\in D(L_1^M)$  and a ball $B\equiv B(x_B,\,r_B)\subset\rn$ such
that
\vspace{-0.25cm}
\begin{enumerate}
\item[(i)] $a=L_1^M b$;
\vspace{-0.25cm}
\item[(ii)] for each $\ell\in\{0,\,1,\,\cdots,\,M\}$,
$\supp L_1^{\ell}b\subset B$;
\vspace{-0.25cm}
\item[(iii)] for all $\ell\in\{0,\,1,\,\cdots,\,M\}$,
\begin{eqnarray}\label{2.2}
\lf\|\lf(r_B^{2}L_1\r)^{k} b\r\|_{L^2(\rn)} \le
r_B^{2M+n(\frac{1}{2}-\frac{1}{p})}.
\end{eqnarray}
\end{enumerate}
\vspace{-0.25cm}

We then have the following atomic decomposition of
$H_{{L_1}}^p(\rn)$.

\begin{theorem}[\cite{hlmmy,jy}]\label{t2.1}
Let $p\in(0,\,1]$. Suppose that $M\in\nn$ and
$M>\frac{n}{2}(\frac{1}{p}-\frac{1}{2})$. Then for all $f\in
L^2(\rn)\cap H_{{L_1}}^p(\rn)$, there exist a sequence
$\{a_j\}_{j=0}^\fz$ of $(p,\,2,\,M)_{L_1}$-atoms and a sequence
$\{\lz_j\}_{j=0}^\fz$ of numbers such that
$f=\sum_{j=0}^\fz\lz_ja_j$ in both $H_{L_1}^p(\rn)$ and $L^2(\rn)$, and
$\|f\|_{H_{L_1}^p(\rn)}\sim\{\sum_{j=0}^\fz|\lz_j|^p\}^{1/p}$.
\end{theorem}

For the second order divergence form operator, the associated Hardy
space were studied in \cite{hm09,hm09c,hmm,jy10}. More precisely,
let $L_{2}\equiv -\rm{div}(A\nabla )$ be a \emph{second order
divergence form elliptic operator with complex bounded measurable
coefficients}. We say that $L_{2}$ is {\it elliptic} if the matrix
$A\equiv\{a_{i,\,j}\}_{i,\,j=1}^n$ satisfying the \emph{elliptic
condition}, namely, there exist positive constants
$0<\lz\le\Lambda<\fz$ such that $\lz|\xi|^2\le \Re
(A\xi\cdot\bar\xi)$ and $|A\xi\cdot\bar\xi|\le\Lambda|\xi|^2$, where
for any $z\in\cc$, $\Re z$ denotes the \emph{real part} of $z$.

\begin{definition}[\cite{hm09,hmm,jy10}]\label{d2.2}
Let $p\in(0,\,1]$ and $L_2$ be the second order divergence form
elliptic operator with complex bounded measurable coefficients. A
function $f\in L^2(\rn)$ is said to be in $\mathbb{H}_{L_2}^p(\rn)$
if $S_{L_2}f\in L^p(\rn)$; moreover, define
$\|f\|_{H_{L_2}^p(\rn)}\equiv\|S_{L_2}f\|_{L^p(\rn)}$. The {\it
Hardy space} $H_{L_2}^p(\rn)$ is then defined to be the completion
of $\mathbb{H}_{L_2}^p(\rn)$ with respect to the quasi-norm
$\|\cdot\|_{H_{L_2}^p(\rn)}$.
\end{definition}

Recall that in \cite{hmm,jy10}, for all $p\in(0,\,1]$, $\ez\in(0,\,\fz)$ and
$M\in\nn$, a function $A\in L^2(\rn)$ is called an
$(H_{L_2}^p,\,\ez,\,M)$-{\it molecule} if there exists a ball
$B\equiv B(x_B,\,r_B)\subset\rn$ such that
\vspace{-0.25cm}
\begin{enumerate}
\item[(i)] for each $\ell\in\{1,\,\cdots,\,M\}$,
$A$ belongs to the range of ${L^\ell_2}$ in $L^2(\rn)$;
\vspace{-0.25cm}
\item[(ii)] for all $i\in\zz_+$ and $\ell\in\{0,\,1,\,\cdots,\,M\}$,
\begin{eqnarray}\label{2.3}
\lf\|\lf(r_B^{2}{L_2}\r)^{-\ell} A\r\|_{L^2(S_i(B))}
\le(2^ir_B)^{n(\frac{1}{2}-\frac{1}{p})}2^{-i\ez},
\end{eqnarray}
where $S_0(B)\equiv B$ and $S_i(B)\equiv 2^iB\setminus 2^{i-1}B$
for all $i\in\nn$.
\end{enumerate}
\vspace{-0.25cm}

Assume that $\{m_j\}_{j}$ is a sequence of
$(H_{L_2}^p,\,\ez,\,M)$-molecules and $\{\lz_j\}_{j}$ a sequence
of numbers satisfying $\sum_j|\lz_j|^p<\fz$. For any $f\in
L^2(\rn)$, if $f=\sum_{j} \lz_jm_j$ in $L^2(\rn)$, then $\sum_{j}
\lz_jm_j$ is called a {\it molecular}
$(H_{L_2}^p,\,2,\,\ez,\,M)$-{\it representation} of $f$. The {\it
molecular Hardy space $H_{{L_2},\,\mol,\,M}^p(\rn)$} is then defined
to be the completion of the space
$$\mathbb{H}_{{L_2},\,\mol,\,M}^p(\rn)\equiv\{f:\ f\ \text{ has a molecular}\
(H_{L_2}^p,\,2,\,\ez,\,M)\text{-representation}\}$$  with respect to
the quasi-norm
\begin{eqnarray*}
\|f\|_{H_{{L_2},\,\mol,\,M}^p(\rn)} \equiv&&\inf\lf\{\lf(\dsum_{j=0
}^\fz|\lz_j|^p\r)^{1/p}:\
f=\dsum_{j=0}^\fz\lz_jA_j \ \text{is a molecular}\r.\\
&&\hspace{2.8cm}(H_{L_2}^p,\,
2,\,\ez,\,M)\text{-representation}\Bigg\},
\end{eqnarray*}
where the infimum is taken over all the molecular $(H_{L_2}^p,\,2,\,
\ez,\,M)$-representations of $f$ as above.

We have the following molecular characterization of
$H_{{L_2}}^p(\rn)$.
\begin{theorem}[\cite{hmm,jy10}]\label{t2.2}
Let $p\in(0, 1]$. Suppose that $M>\frac{n}{2}(\frac{1}{p}-\frac{1}{2})$
and $\ez>0$. Then $H_{L_2}^p(\rn)=H_{{L_2},\,\mol,\,M}^p(\rn)$.
Moreover, $\|f\|_{H_{L_2}^p(\rn)}\sim\|f\|_{H_{{L_2},\,\mol,\,M}^p(\rn)}$,
where the implicit constants depend only on $M,\,n,\,p,\,\ez$ and
the constants appearing in the ellipticity.
\end{theorem}

We now recall the definition of the weak Hardy space (see, for example,
\cite{fs1,li,lu}). Let $p\in(0,\,1]$ and $\fai\in\mathcal{S}(\rn)$ with
support in the unit ball $B(0,\,1)$. The {\it weak Hardy space}
$WH^p(\rn)$ is defined to be the space
$$\lf\{f\in\mathcal{S}'(\rn):\ \|f\|_{WH^p(\rn)}\equiv
\sup_{\az>0}\lf(\az^p\lf|\lf\{x\in\rn:\ \sup_{t>0}\lf|\fai_t\ast
f(x)\r|>\az\r\}\r|\r)^{1/p}<\fz\r\}.$$

Let $L_1$ be a \emph{nonnegative self-adjoint operator} in $L^2(\rn)$
satisfying the assumptions (A$_1$) and (A$_2$). Following
\cite{al}, let the operator $D$ be a \emph{linear operator} defined densely
in $L^2(\rn)$ and satisfy the following \emph{assumptions}:

(B$_1$) $DL_1^{-1/2}$ is bounded on $L^2(\rn)$;

(B$_2$) the family of operators, $\{\sqrt{t}De^{-t{L_1}}\}_{t>0}$,
satisfy the Davies-Gaffney estimates as in \eqref{2.1};

(B$_3$) for all $(p,\,2,\,M)_{L_1}$-atoms $a$,
$\int_{\rn}DL_1^{-1/2}a(x)\,dx=0$.

Typical examples of $D$ and $L_1$ satisfying the assumptions (B$_1$),
(B$_2$) and (B$_3$) include that $D$ is the gradient operator $\nabla$
on $\rn$, and $L_1$ is the second order divergence form
elliptic operator with real symmetric bounded
measurable coefficients or the Schr\"odinger operator $-\bdz+V$
with $0\le V\in L^1_\loc(\rn)$ as proved below.

\begin{lemma}\label{l2.1}
Let $0\le V\in L_{\loc}^1(\rn)$. Then the Schr\"odinger operator
$T\equiv-\bdz+V$ satisfies the
assumptions (A$_1$) and (A$_2$), and both $T$ and the gradient operator $\nabla$
satisfy the assumptions (B$_1$), (B$_2$) and (B$_3$).
\end{lemma}

\begin{proof}
It is easy to see that $T$ is nonnegative
self-adjoint.

Let $e^{-t{T}}(\cdot,\,\cdot)$ be the \emph{integral kernel}
of the semigroup $e^{-t{T}}$. By Trotter's formula (see, for example,
\cite{tr}), we know that for all $t\in(0,\,\fz)$ and $x,\,y\in\rn$,
\begin{eqnarray*}
0\le e^{-t{T}}(x,\,y)\le
e^{-t\bdz}(x,\,y)\sim t^{-\frac{n}{2}}\exp\lf\{-\frac{|x-y|^2}{t}\r\},
\end{eqnarray*}
which implies that the semigroup $\{e^{-t{T}}\}_{t>0}$ satisfies
\eqref{2.1}. Thus, ${T}$ satisfies the assumptions (A$_1$) and
(A$_2$).

Moreover, by \cite[Lemma 8.5]{hlmmy}, we conclude that there
exists a positive constant $C_2$ such that for for all
closed sets $E$, $F\subset\rn$, $t\in(0,\,\fz)$ and $f\in L^2(\rn)$
supported in $E$,
\begin{eqnarray*}
\lf\|t\nabla e^{-t^2T}f\r\|_{L^2(F)}\ls
\exp\lf\{-\frac{[\dist(E,\,F)]^2}{C_2t^2}\r\}\|f\|_{L^2(E)},
\end{eqnarray*}
which, combining the $L^2(\rn)$-boundedness of the Riesz transform
$\nabla(T^{-1/2})$ (see \cite[(8.20)]{hlmmy}) and the fact that
$\int_{\rn}\nabla(T^{-1/2})a(y)\,dy=0$ (see, for example
\cite{hlmmy,jy}), implies that both $T$ and the gradient operator satisfy the
assumptions (B$_1$), (B$_2$) and (B$_3$).
This finishes the proof of Lemma \ref{l2.1}.
\end{proof}

We also need the following technical lemmas.
\begin{lemma}[\cite{lu,stw}]\label{l2.2}
Let $p\in(0,\,1)$ and $\{f_j\}_j$ be a sequence of measurable
functions. If $\sum_{j}\lf|\lz_j\r|^p<\fz$ and there exists a
positive constant $\wz C$ such that for all $\{f_j\}_j$ and $\az\in (0,\fz)$,
$|\{x\in\rn:\ |f_j|>\az\}|\le \wz C\az^{-p}$. Then, for all $\az\in(0,\fz)$,
\begin{eqnarray*}
\lf|\lf\{x\in\rn:\ \lf|\dsum_{j}\lz_jf_j(x)\r|>\az\r\}\r|\le
\wz C\frac{2-p}{1-p}\az^{-p}\sum_{j}|\lz_j|^p.
\end{eqnarray*}
\end{lemma}

\begin{lemma}[\cite{al,hm}]\label{l2.3}
Let $L_1$ be a nonnegative self-adjoint operator satisfying the
assumptions (A$_1$) and (A$_2$) and $D$
the operator satisfying the assumptions (B$_1$), (B$_2$) and (B$_3$).
Let $M\in\nn$. Then there exists a positive constant $C$, depending
on $M$, such that for all closed sets $E$, $F$ in $\rn$ with $\dist(E,\,F)>0$,
$f\in L^2(\rn)$ supported in $E$ and $t\in(0,\,\fz)$,
\begin{eqnarray}\label{2.4}
\lf\|DL_1^{-1/2}\lf(I-e^{-t{L_1}}\r)^{M}f\r\|_{L^2(F)}\le C
\lf(\frac{t}{\lf[\dist(E,\,F)\r]^{2}}\r)^M\|f\|_{L^2(E)}
\end{eqnarray}
and
\begin{eqnarray}\label{2.5}
\lf\|DL_1^{-1/2}\lf(t{L_1}e^{-t{L_1}}\r)^{M}f\r\|_{L^2(F)}\le C
\lf(\frac{t}{\lf[\dist(E,\,F)\r]^{2}}\r)^M\|f\|_{L^2(E)}.
\end{eqnarray}
Moreover, if $L_2$ is a second order divergence form elliptic
operator with complex bounded measurable coefficients, then
\eqref{2.4} and \eqref{2.5} still hold when $D$ and $L_1$ are
replaced, respectively, by the gradient operator $\nabla$ and $L_2$.
\end{lemma}

\section{Proofs of main results}\label{s3}

\hskip\parindent In this section, we show Theorem \ref{t1.1}, Corollary
\ref{c1.1} and Theorem \ref{t1.2}.

\begin{proof}[Proof of Theorem \ref{t1.1}]
Let $p\equiv\frac{n}{n+1}$. By the density of $H^p_{L_1}(\rn)\cap L^2(\rn)$
in $H^p_{L_1}(\rn)$, we only need consider $f\in H^p_{L_1}(\rn)\cap L^2(\rn)$.
Let $M\in\nn$ and $M>\max\{\frac 12+\frac n4,\,1\}$. By Theorem \ref{t2.1},
we know that there exist a sequence
$\{a_j\}_j$ of $(p,\,2,\,M)_{L_1}$-atoms and a sequence
$\{\lz_j\}_j$ of numbers such that
\begin{equation}\label{3.1}
f=\sum_j\lz_j a_j\
\end{equation}
in $L^2(\rn)$ and $\|f\|_{H^p_{L_1} (\rn)}\sim\{\sum_{j}|\lz_j|^p\}^{1/p}$. To
show Theorem \ref{t1.1}, by \eqref{3.1} and the definition
of $WH^p(\rn)$, we see that it
suffices to prove that for all $\az\in(0,\fz)$,
\begin{equation}\label{3.2}
\lf|\lf\{x\in\rn:\ \sup_{0<t<\fz}\lf|\fai_t\ast\lf(\sum_j\lz_jD
L_1^{-1/2}a_j\r)(x)\r|>\az\r\}\r|\ls\frac 1{\az^p}\sum_{j}|\lz_j|^p,
\end{equation}
where $\fai\in C^{\fz}_{c}(\rn)$ satisfies $\supp\fai\subset
B(0,1)$, and for all $x\in\rn$ and $t\in(0,\fz)$,
$\fai_t(x)\equiv\frac{1}{t^n}\fai(\frac{x}{t})$. In order to prove
\eqref{3.2}, by Lemma \ref{l2.2}, it suffices to show that for any
$(p,\,2,\,M)_{L_1}$-atom $a$ associated with the ball $B\equiv B(x_B,r_B)$
and $\az\in(0,\fz)$,
\begin{equation*}
\lf|\lf\{x\in\rn:\ \sup_{0<t<\fz}\lf|\fai_t\ast\lf(D
L_1^{-1/2}a\r)(x)\r|>\az\r\}\r|\ls\frac{1}{\az^p}.
\end{equation*}

Let $\mathcal{M}$ be the \emph{Hardy-Littlewood maximal function}. It is
easy to see that
$$\sup_{0<t<\fz}\lf|\fai_t\ast(D
L_1^{-1/2}a)\r|\ls\mathcal{M}(D L_1^{-1/2}a).$$
Then by Chebyshev's inequality, H\"older's inequality, the
$L^2(\rn)$-boundedness of $\mathcal{M}$, the $L^2(\rn)$-boundedness
of $D L_1^{-1/2}$ via (B$_1$), and \eqref{2.2}, we know that
\begin{eqnarray*}
&&\lf|\lf\{x\in 16B:\ \sup_{0<t<\fz}\lf|\fai_t\ast\lf(D
L_1^{-1/2}a\r)(x)\r|>\az\r\}\r|\\ \nonumber
&&\hs\ls\frac{1}{\az^p}\lf\|\sup_{0<t<\fz}\lf|\fai_t\ast\lf(D
L_1^{-1/2}a\r)\r|\r\|^p_{L^p(16B)}
\ls\frac{1}{\az^p}\lf\|\cm\lf(D
L_1^{-1/2}a\r)\r\|^p_{L^p(16B)}\\ \nonumber
&&\hs\ls\frac{1}{\az^p}\lf\|\cm\lf(D
L_1^{-1/2}a\r)\r\|^p_{L^2(\rn)}|B|^{1-\frac{p}{2}}
\hs\ls\frac{1}{\az^p}\|a\|^p_{L^2(\rn)}|B|^{1-\frac{p}{2}}\ls
\frac{1}{\az^p}.
\end{eqnarray*}
On the other hand, we have
\begin{eqnarray*}
&&\lf\{x\in (16B)^\complement:\
\sup_{0<t<\fz}\lf|\fai_t\ast\lf(D L_1^{-1/2}a\r)(x)\r|>\az\r\}\\
\nonumber &&\hs\subset\lf\{x\in (16B)^\complement:\
\sup_{0<t<r_B}\lf|\fai_t\ast\lf(D L_1^{-1/2}a\r)(x)\r|>\az/2\r\}\\
\nonumber&&\hs\hs\bigcup\lf\{x\in
(16B)^\complement:\ \sup_{r_B<t<\fz}\lf|\cdots\r|>\az/2\r\}
\equiv \rm{I}\cup\rm{J}.
\end{eqnarray*}

To estimate $\rm{I}$, let $S_i(B)\equiv 2^iB\setminus2^{i-1}B$ and
$\wz{S}_i(B)\equiv 2^{i+1}B\setminus2^{i-2}B$ with $i\in\nn$. For
all $i\ge5$, $x\in S_i(B)$ and $y\in B(x,\,r_B)$, from
$\supp \fai\subset B(0,\,1)$, it follows that $y\in\wz{S}_i(B)$. For
$i\ge5$, let
$$\mathrm{I}_i\equiv\lf\{x\in S_i(B):\ \sup_{0<t<r_B}\lf|\fai_t\ast
\lf(D L_1^{-1/2}a\r)(x)\r|>\az/2\r\}.$$
By Chebyshev's inequality, H\"older's inequality, the $L^2(\rn)$-boundedness of
$\cm$, Lemma \ref{l2.3} and \eqref{2.2}, we conclude that
\begin{eqnarray*}
|\mathrm{I}_i|&&\ls\az^{-p}\dint_{S_i(B)}\lf[\sup_{0<t<r_B}\lf|
\dint_{\wz{S}_i(B)}t^{-n}\fai \lf(\frac{x-y}{t}\r)\lf[\chi_{\wz
{S}_i(B)}(y)D
L_1^{-1/2}a(y)\r]\,dy\r|\r]^p\,dx\\
&&\ls\az^{-p}\dint_{S_i(B)}\lf[\cm\lf(\chi_{\wz {S}_i(B)}D
L_1^{-1/2}a\r)(x)\r]^p\,dx\\&&\ls\az^{-p}|S_i(B)|^{1-p/2}\lf\|D
L_1^{-1/2}a\r\|_{L^2(\wz {S}_i(B))}^{p}\\
&&\ls\az^{-p}|S_i(B)|^{1-p/2}\lf[\lf\|D
L_1^{-1/2}\lf(I-e^{-r_B^2{L_1}}\r)^Ma\r\|_{L^2(\wz
{S}_i(B))}^{p}\r.\\
&&\lf.\hs+\dsum_{k=1}^M\lf\|D
L_1^{-1/2}\lf(r_B^2L_1e^{-\frac{k}{M}r_B^2{L_1}}\r)^Mr_B^{-2M}b\r\|_{L^2(\wz
{S}_i(B))}^{p}\r]\\
&&\ls\az^{-p}|S_i(B)|^{1-p/2}\lf[\frac{r_B^2}{(2^ir_B)^2}\r]^{Mp}|B|^{p/2-1}
\sim2^{-i[2Mp-n(1-p/2)]}\az^{-p}.
\end{eqnarray*}
From this, the definition of $\mathrm{I}_i$, $p=\frac{n}{n+1}$ and
$M>\frac 12+\frac n4$, we deduce that
$|\mathrm{I}|\ls\sum_{i=1}^\fz|\mathrm{I}_i|\ls\frac{1}{\az^p}$,
which is a desired estimate for $\mathrm{I}$.

To estimate $\mathrm{J}$,  by the assumption that $\int_{\rn}D
L_1^{-\frac{1}{2}}a(y)\,dy=0$ via (B$_3$), we know that
\begin{eqnarray*}
|\mathrm{J}|\ls&&\Bigg|\Bigg\{x\in (16B)^\complement:\\
&&\quad \dsum_{i=0}^\fz \sup_{r_B<t<\fz}\lf.\lf.
\lf|\int_{S_i(B)}\frac{1}{t^n}\lf[\fai \lf(\frac{x-y}{t}\r)-\fai
\lf(\frac{x-x_B}{t}\r)\r]D
L_1^{-\frac{1}{2}}a(y)\,dy\r|>\az/2\r\}\r|.
\end{eqnarray*}
Let $F_i(x)\equiv\sup_{r_B<t<\fz} |\int_{S_i(B)}\frac{1}{t^n}[\fai
(\frac{x-y}{t})-\fai (\frac{x-x_B}{t})]D
L_1^{-\frac{1}{2}}a(y)\,dy|$ and
$$\mathrm{J}_i\equiv\lf\{x\in
(16B)^\complement: \ F_i(x)>\az/2\r\}.$$
To obtain a desired estimate for $\mathrm{J}$, by Lemma \ref{l2.2},
it suffices to show that there
exists a positive constant $C_0$ such that
\begin{eqnarray}\label{3.3}
|\mathrm{J}_i|\ls\frac{2^{-C_0i}}{\az^p}.
\end{eqnarray}
From the mean value theorem, H\"older's inequality, $\supp
\fai\subset B(0,\,1)$, Lemma \ref{l2.3} and \eqref{2.2}, we infer
that
\begin{eqnarray*}
F_i(x)&&\le\sup_{j\in\zz_+}\sup_{2^jr_B\le
t<2^{j+1}r_B}\chi_{(2^{i+1}+2^{j+1})B}(x)\int_{S_i(B)}\frac{1}{t^n}\|\nabla
\fai\|_{L^\fz(\rn)}\lf|\frac{y-x_B}{t}\r|\lf|D
L_1^{-\frac{1}{2}}a(y)\r|\,dy\\
&&\ls\sup_{j\in\zz_+}\chi_{(2^{i+1}+2^{j+1})B}(x)\sup_{2^jr_B\le
t<2^{j+1}r_B}2^{-j(n+1)} |B|^{-1}2^{i}|S_i(B)|^{1/2}\\
&&\hs\times \|D
L_1^{-\frac{1}{2}}a\|_{L^2(S_i(B))}\\
&&\ls\sup_{j\in\zz_+}\chi_{(2^{i+1}+2^{j+1})B}(x)\sup_{2^jr_B\le
t<2^{j+1}r_B}2^{-j(n+1)}
2^{i(n/2+1)}\lf[\frac{r_B^2}{(2^ir_B)^2}\r]^M|B|^{-1/p}
\\
&&\equiv C_3\sup_{j\in\zz_+}\chi_{(2^{i+1}+2^{j+1})B}(x)\sup_{2^jr_B\le
t<2^{j+1}r_B}2^{-j(n+1)} 2^{-i(2M-n/2-1)}|B|^{-1/p}.
\end{eqnarray*}

Let
$$j_0\equiv\max\lf\{j\in\zz_+:\ C_32^{-j(n+1)}
2^{-i(2M-n/2-1)}|B|^{-1/p}>\az/2\r\}.$$
For all $x\in [(2^{i+1}+2^{j_0+1})B]^\complement$, we see that
\begin{eqnarray*}
F_i(x)\le C_3\sup_{j\ge
j_0}\chi_{(2^{i+1}+2^{j+1})B}(x)\sup_{2^jr_B\le
t<2^{j+1}r_B}2^{-j(n+1)} 2^{-i(2M-n/2-1)}|B|^{-1/p}\le\az/2,
\end{eqnarray*}
which implies that $x\in\mathrm{J}_i^{\complement}$. Thus,
$\mathrm{J}_i\subset(2^{i+1}+2^{j_0+1})B$. From this and Chebyshev's
inequality, we then deduce that
\begin{eqnarray*}
|\mathrm{J}_i|\ls\az^{-p}\dint_{(2^{i+1}+2^{j_0+1})B}
2^{-pj_0(n+1)}2^{-ip(2M-1+n)}|B|^{-1}dx\ls2^{-i[(2M-1)p-n(1-p)]}\az^{-p},
\end{eqnarray*}
which implies that \eqref{3.3} holds with $C_0\equiv(2M-1)p-n(1-p)$.
Observe that $C_0>0$, since $M>1$ and $p=\frac{n}{n+1}$. Thus,
combining the estimate of $\mathrm{I}$ and $\mathrm{J}$, we then
complete the proof of Theorem \ref{t1.1}.
\end{proof}

\begin{proof}[Proof of Corollary \ref{c1.1}]
From Lemma \ref{l2.1}, we deduce that the Schr\"odinger operator
$-\bdz+V$ with $0\le V\in L^1_{\loc}(\rn)$ satisfies the assumptions
(A$_1$) and (A$_2$) as in Section \ref{s2},  and both $-\bdz+V$ and
the gradient operator $\nabla$ satisfy the assumptions (B$_1$),
(B$_2$) and (B$_3$) as in Section \ref{s2}. Thus, from Theorem
\ref{t1.1}, we deduce that the Riesz transform
$\nabla(-\bdz+V)^{-1/2}$ is bounded from $H_{-\bdz+V}^p(\rn)$ to the
classical weak Hardy space $WH^p(\rn)$ in the critical case that
$p=n/(n+1)$, which completes the proof of Corollary \ref{c1.1}.
\end{proof}

\begin{proof}[Proof of Theorem \ref{t1.2}]
Let $p=\frac{n}{n+1}$ and $M\in\nn$ satisfy
$M>\frac{n}{4}+\frac{1}{2}$.
To prove Theorem \ref{t1.2}, similar to the proof of Theorem
\ref{t1.1}, by Theorem \ref{t2.2} and Lemma \ref{l2.2},
for each $(H_L^p,\,\ez,\,M)$-molecule $A$ associated to
the ball $B(x_B,\,r_B)$, $m\in\zz_{+}$ and $\az\in(0,\,\fz)$, we
only need estimate the measure of the following sets:
\begin{eqnarray*}
\wz{\rm{I}}\equiv\lf\{x\in (16B)^\complement: \
\sup_{0<t<r_B}\lf|\fai_{t}*(\nabla  L_2^{-1/2}A)(x)\r|>\az/2\r\}
\end{eqnarray*}
and
\begin{eqnarray*}
\wz{\rm{J}}\equiv\lf\{x\in (16B)^\complement: \ \sup_{r_B\le
t<\fz}\lf|\fai_{t}*\lf(\nabla L_2^{-1/2}A\r)(x)\r|>\az/2\r\}.
\end{eqnarray*}
The estimate of $\mathrm{\wz{I}}$ is similar to that of $\mathrm{I}$
in the proof of Theorem \ref{t1.1}. We omit the details. Now we
estimate $\mathrm{\wz{J}}$. Since
\begin{eqnarray*}
|\wz{\mathrm{J}}|&&\ls\Bigg|\Bigg\{x\in (16B)^\complement:\
\dsum_{i=0}^\fz \sup_{r_B\le t<\fz}\lf.\lf.
\lf|\int_{S_i(B)}\frac{1}{t^n}\lf[\fai \lf(\frac{x-y}{t}\r)-\fai
\lf(\frac{x-x_B}{t}\r)\r]\r.\r.\r.\\
&&\hs\times\nabla
L_2^{-\frac{1}{2}}\lf(I-e^{-r_B^2L_2}\r)^MA(y)\,dy\Bigg|>\az/2\Bigg\}\Bigg|\\
&&\hs+\Bigg|\Bigg\{x\in (16B)^\complement:\ \dsum_{i=0}^\fz
\sum_{k=1}^M\sup_{r_B\le t<\fz}\lf.\lf. \lf|\int_{S_i(B)}\frac{1}{t^n}\lf[\fai
\lf(\frac{x-y}{t}\r)-\fai
\lf(\frac{x-x_B}{t}\r)\r]\r.\r.\r.\\
&&\hs\times\nabla L_2^{-\frac{1}{2}}(r_B^2L_2e^{-\frac{k}
{M}r_B^2L_2})^M(r_B^2L_2)^{-M}A(y)\,dy\Bigg|>\az/2\Bigg\}\Bigg|.
\end{eqnarray*}
Let $\wz{F_{1,i}}(x)\equiv \sup_{r_B\le
t<\fz}|\int_{S_i(B)}\frac{1}{t^n}[\fai
(\frac{x-y}{t})-\fai(\frac{x-x_B}{t})]\nabla
L_2^{-\frac{1}{2}}(I-e^{-r_B^2L_2})^MA(y)\,dy|$,
\begin{eqnarray*}
\wz{F_{2,i}}(x)\equiv&&\sum^M_{k=1}\sup_{r_B\le
t<\fz}\Bigg|\int_{S_i(B)}\frac{1}{t^n}\lf[\fai
\lf(\frac{x-y}{t}\r)-\fai\lf(\frac{x-x_B}{t}\r)\r]\\
&&\times\nabla L_2^{-\frac{1}{2}}\lf(r_B^2L_2
e^{-\frac{k}{M}r_B^2L_2}\r)^M(r_B^2L_2)^{-M}A(y)\,dy\Bigg|,
\end{eqnarray*}
$\wz{\mathrm{J}_{1,k}}\equiv\{x\in(16B)^\complement: \
\wz{F_{1,i}}(x)>\az/2\}$ and
$\wz{\mathrm{J}_{2,k}}\equiv\{x\in(16B)^\complement: \
\wz{F_{2,i}}(x)>\az/2\}$. By Lemma \ref{l2.2}, it suffices to show
that there exist positive constants $C_4$ and $C_5$  such that for
all $\az\in(0,\,\fz)$,
$|\wz{\mathrm{J}_{1,k}}|\ls\frac{2^{-C_4i}}{\az^p}$ and
$|\wz{\mathrm{J}_{2,k}}|\ls\frac{2^{-C_5i}}{\az^p}$. We only prove
the first inequality, the proof of the second inequality is similar.
Take $\ez\in(n+1-1/(n+1),\,\fz)$.  By the mean value theorem,
H\"older's inequality, Lemma \ref{l2.3}, \eqref{2.3} and $\supp
\fai\subset B(0,\,1)$, we conclude that
\begin{eqnarray*}
\wz{F_{1,i}}(x)\ls&&\sup_{j\in\zz_+}\chi_{(2^{i+1}+2^{j+1})B}(x)\sup_{2^jr_B\le
t<2^{j+1}r_B}\int_{S_i(B)}\frac{1}{t^n}\|\nabla
\fai\|_{L^\fz(\rn)}\lf|\frac{y-x_B}{t}\r|\\
&&\times\lf|\nabla
L_2^{-\frac{1}{2}}\lf(I-e^{-r_B^2L_2}\r)^MA(y)\r|\,dy\\
\ls&&\sup_{j\in\zz_+}\chi_{(2^{i+1}+2^{j+1})B}(x)\sup_{2^jr_B\le
t<2^{j+1}r_B}\int_{S_i(B)}\frac{1}{t^n}\|\nabla
\fai\|_{L^\fz(\rn)}\lf|\frac{y-x_B}{t}\r|\\
&&\times\lf|\nabla
L_2^{-\frac{1}{2}}\lf(I-e^{-r_B^2L_2}\r)^M(\chi_{\wz{S}_i(B)}A)(y)\r|\,dy \\
&&+\sup_{j\in\zz_+}\chi_{(2^{i+1}+2^{j+1})B}(x)\sup_{2^jr_B\le
t<2^{j+1}r_B}\int_{S_i(B)}\frac{1}{t^n}\|\nabla
\fai\|_{L^\fz(\rn)}\lf|\frac{y-x_B}{t}\r|\\
&&\times \lf|\nabla L_2^{-\frac{1}{2}}\lf(I-e^{-r_B^2L_2}\r)^M
(\chi_{\rn\setminus\wz{S}_i(B)}A)(y)\r|\,dy\\
\ls&&\sup_{j\in\zz_+}\chi_{(2^{i+1}+2^{j+1})B}(x)\\
&&\times\sup_{2^jr_B\le
t<2^{j+1}r_B}2^{-j(n+1)}
\lf[2^{-i(\ez+n/p-n-1)}+2^{-i(2M-n/2-1)}\r]|B|^{-1/p},
\end{eqnarray*}
where ${S}_i(B)$ and $\wz{S}_i(B)$ are as in the proof of Theorem
\ref{t1.1}. The rest of the proof is similar to that of Theorem
\ref{t1.1}; we omit the details. This finishes the proof of Theorem
\ref{t1.2}.
\end{proof}

\section{Further remarks}\label{s4}

\hskip\parindent In this section, we establish a variant of Theorems
\ref{t1.1} and \ref{t1.2} for the higher order divergence form
elliptic operators with complex bounded measurable coefficients and
the higher order Schr\"odinger-type operators.

To this end, we first
recall some notion and notations. For $\tz\in[0,\,\pi)$, the {\it
closed sector, $S_\tz$, of angle $\tz$} in the complex plane $\cc$
is defined by $S_\tz\equiv\lf\{z\in\cc\setminus\{0\}:\ |\arg
z|\le\tz\r\}\cup\lf\{0\r\}$. Let $\omega\in[0,\,\pi)$. A closed
operator $T$ in $L^2(\rn)$ is called of {\it type} $\omega$ (see,
for example, \cite{mc}), if its spectrum, $\sz(T)$, is contained in
$S_\omega$, and for each $\tz\in (\omega,\,\pi)$, there exists a
nonnegative constant $C$ such that for all $z\in\cc\setminus S_\tz$,
$\|(T-zI)^{-1}\|_{\cl(L^2(\rn))}\le C|z|^{-1}$, where and in what
follows, $\|S\|_{\cl(\mathcal H)}$ denotes the {\it operator norm}
of the linear operator $S$ on the normed linear space $\mathcal H$.
Let $T$ be a one-to-one operator of type $\omega$, with
$\omega\in[0,\,\pi)$ and $\mu\in(\omega,\,\pi)$, and $f\in
H_{\fz}(S_\mu^0)\equiv \{f\ \text{is holomorphic on}\ S_\mu^0:\
\|f\|_{L^{\fz}(S^0_{\mu})}<\fz\}$, where $S_\mu^0$ denotes the
\emph{interior} of $S_\mu$.  By the $H_\fz$ functional calculus, the
function of the operator $T$, $f(T)$ is well defined. The operator
$T$ is said to have a {\it  bounded $H_\fz$ functional calculus} in
the Hilbert space $\mathcal{H}$, if there exist $\mu\in (0,\,\pi)$
and positive constant $C$ such that for all $\psi\in
H_\fz(S_\mu^0)$, $\|\psi(T)\|_{\cl(\mathcal{H})}\le
C\|\psi\|_{L^\fz(S_\mu^0)}$.

As in \cite{cy},  let $T$ be an operator defined in $L^2(\rn)$ which
satisfies the following \emph{assumptions}: \vspace{-0.25cm}
\begin{enumerate}
\item[(E$_1$)]
 The operator $T$ is a one-to-one operator of
type $\omega$ in $L^2(\rn)$ with $\omega\in[0,\,\pi/2)$;
\vspace{-0.25cm}
\item[(E$_2$)] The operator $T$ has a bounded $H_\fz$
functional calculus in $L^2(\rn)$; \vspace{-0.25cm}
\item[(E$_3$)] Let $k\in \nn$. The operator $T$
generates a holomorphic semigroup $\{e^{-tT}\}_{t>0}$ which
satisfies the {\it $k$-Davies-Gaffney estimate}, namely, there exist
positive constants $C_6$ and $C_7$ such that for all closed sets $E$
and $F$ in $\rn$, $t\in(0,\,\fz)$ and $f\in L^2(\rn)$ supported in
$E$,
\begin{eqnarray*}
\|e^{-tT}f\|_{L^2(F)}\le C_6
\exp\lf\{-\frac{\lf[\dist(E,\,F)\r]^{2k/(2k-1)}}
{C_7t^{1/(2k-1)}}\r\}\|f\|_{L^2(E)}.
\end{eqnarray*}
\end{enumerate}
\vspace{-0.25cm}

When $k=1$, the $k$-Davies-Gaffney estimate is just \eqref{2.1}.

Let $k\in\nn$. Typical examples of operators, satisfying the above
assumptions (E$_1$), (E$_2$) and (E$_3$), include the following
\emph{$2k$-order divergence form homogeneous elliptic operator}
\begin{equation}\label{4.1}
T_1\equiv(-1)^k\dsum_{|\az|=|\bz|=k}\pat^\az(a_{\az,\bz}\pat^\bz)
\end{equation}
with complex bounded measurable coefficients
$\{a_{\az,\,\bz}\}_{|\az|=|\bz|=k}$, and the following
\emph{$2k$-order Schr\"odinger-type operator}
\begin{equation}\label{4.2}
T_2\equiv(-\bdz)^k+V^k
\end{equation}
with $0\le V\in L^k_{\loc}(\rn)$.

For all $f\in L^2(\rn)$ and $x\in\rn$, define the \emph{$T$-adapted
square function} $S_Tf(x)$ by
\begin{eqnarray*}
S_Tf(x)\equiv\lf\{\iint_{\Gamma(x)}|t^{2k}Te^{-t^{2k}T}f(y)|^2
\frac{dy\,dt}{t^{n+1}}\r\}^{1/2}.
\end{eqnarray*}

Using the $T$-adapted square function $S_Tf$, Cao and Yang \cite{cy}
introduced the following Hardy space $H_T^p(\rn)$ associated to $T$.

\begin{definition}[\cite{cy}]\label{d4.1}
Let $p\in(0,\,1]$ and $T$ satisfy the assumptions (E$_1$), (E$_2$)
and (E$_3$). A function $f\in L^2(\rn)$ is said to be in
$\mathbb{H}_T^p(\rn)$ if $S_Tf\in L^p(\rn)$; moreover, define
$\|f\|_{H_T^p(\rn)}\equiv\|S_Tf\|_{L^p(\rn)}.$ The {\it Hardy space}
$H_T^p(\rn)$ is then defined to be the completion of
$\mathbb{H}_T^p(\rn)$ with respect to the quasi-norm
$\|\cdot\|_{H_T^p(\rn)}$.
\end{definition}

Let $i\in\{1,\,2\}$. By first establishing the molecular
characterization of $H_{T_i}^p(\rn)$, Cao and Yang \cite{cy} then
obtain the following boundedness of the Riesz transform $\nabla^k
(T_i^{-1/2})$ from $H_{T_i}^p(\rn)$ to $H^p(\rn)$ when
$p\in(n/(n+k),\,1]$.

\begin{theorem}[\cite{cy}]\label{t4.1}
Let $k\in\nn$, $p\in(n/(n+k),\,1]$, $T_1$ be the $2k$-order
divergence form homogeneous elliptic operator with complex bounded
measurable coefficients as in \eqref{4.1}, and $T_2$ the
$2k$-order Schr\"odinger-type operator as in \eqref{4.2}.
Then, for $i\in\{1,\,2\}$, the Riesz transform $\nabla^k
(T_i^{-1/2})$ is bounded from $H_{T_i}^p(\rn)$ to $H^p(\rn)$.
\end{theorem}
Again, for $i\in\{1,\,2\}$, applying the  molecular characterization
of $H_{T_i}^p(\rn)$ from \cite{cy}, by an argument similar to that
used in the proof of Theorem \ref{t1.2}, we obtain the endpoint
boundedness of $\nabla^k (T_i^{-1/2})$ in the critical case that
$p=n/(n+k)$. We omit the details by similarity.

\begin{theorem}\label{t4.2}
Let $k\in\nn$, $p\equiv n/(n+k)$, $T_1$ be the $2k$-order divergence
form homogeneous elliptic operator with complex bounded measurable
coefficients as in \eqref{4.1}, and $T_2$ the $2k$-order
Schr\"odinger-type operator as in \eqref{4.2}.
Then, for $i\in\{1,\,2\}$, the Riesz transform $\nabla^k
(T_i^{-1/2})$ is bounded from $H_{T_i}^p(\rn)$ to $WH^p(\rn)$.
\end{theorem}

\noindent{\bf Acknowledgements.} The second author is supported by
National Natural Science Foundation (Grant No. 11171027) of China
and Program for Changjiang Scholars and Innovative Research Team in
University of China.

\end{document}